\documentclass[11pt]{article}
\usepackage{amsthm, amsmath, amssymb, amsfonts, url, booktabs, tikz, setspace, fancyhdr, enumerate}
\usepackage[margin = 1in]{geometry}

\usepackage[czech,english]{babel} % This will make Stovicek look correct

\usepackage{hyperref}

\usepackage{cleveref}

\crefname{equation}{}{}
\Crefname{equation}{Equation}{Equations}
\crefname{theorem}{Theorem}{Theorems}
\Crefname{theorem}{Theorem}{Theorems}
\crefname{lemma}{Lemma}{Lemmas}
\Crefname{lemma}{Lemma}{Lemmas}
\crefname{proposition}{Proposition}{Propositions}
\Crefname{proposition}{Proposition}{Propositions}
\crefname{corollary}{Corollary}{Corollaries}
\Crefname{corollary}{Corollary}{Corollaries}
\crefname{conjecture}{Conjecture}{Conjectures}
\Crefname{conjecture}{Conjecture}{Conjectures}
\crefname{section}{Section}{Sections}
\Crefname{section}{Section}{Sections}
\crefname{example}{Example}{Examples}
\Crefname{example}{Example}{Examples}
\crefname{problem}{Problem}{Problems}
\Crefname{problem}{Problem}{Problems}
\crefname{table}{Table}{Tables}
\Crefname{table}{Table}{Tables}
\crefname{remark}{Remark}{Remarks}
\Crefname{remark}{Remark}{Remarks}
\crefname{definition}{Definition}{Definitions}
\Crefname{definition}{Definition}{Definitions}

\definecolor{refkey}{gray}{.75}
\definecolor{labelkey}{gray}{.2}

\usepackage[textsize=scriptsize,colorinlistoftodos]{todonotes}

\newtheorem{theorem}{Theorem}[section]
\newtheorem{proposition}[theorem]{Proposition}
\newtheorem{lemma}[theorem]{Lemma}
\newtheorem{corollary}[theorem]{Corollary}

\newtheorem{question}[theorem]{Question}

\theoremstyle{definition}

\newtheorem{example}[theorem]{Example}

\theoremstyle{remark}

\newcommand{\ZZ}{\mathbb{Z}}

\newcommand{\ext}{\mathrm{ex}}

\newcommand{\mockalph}[1]{}

\tikzstyle{p}+=[fill=black, circle, minimum width = 1pt, inner sep =
1pt]
\tikzstyle{w}+=[fill=white, draw, circle, minimum width = 1pt, inner sep =
1.5pt]

\title{Extremal numbers and Sidorenko's conjecture}
\author{David Conlon\thanks{Department of Mathematics, California Institute of Technology, Pasadena, CA 91125, USA. Email: {\tt dconlon@caltech.edu}.
Research supported by NSF Award DMS-2054452.} 
\and 
Joonkyung Lee\thanks{Department of Mathematics, Yonsei University, Seoul, South Korea. Email: \texttt{joonkyunglee@yonsei.ac.kr}. Research supported by Samsung STF Grant SSTF-BA2201-02 and the Yonsei University Research Fund 2023-22-0125.} 
\and 
Alexander Sidorenko\thanks{Alfr\'ed R\'enyi Institute of Mathematics, Budapest, Hungary. Email: \texttt{sidorenko.ny@gmail.com}.}}
\date{}

\begin{document}
\maketitle

\begin{abstract}
Sidorenko's conjecture states that, for all bipartite graphs $H$, quasirandom graphs contain asymptotically the minimum number of copies of $H$ taken over all graphs with the same order and edge density. While still open for graphs, the analogous statement is known to be false for hypergraphs. We show that there is some advantage in this, in that if Sidorenko's conjecture does not hold for a particular $r$-partite $r$-uniform hypergraph $H$, then it is possible to improve the standard lower bound, coming from the probabilistic deletion method, for its extremal number $\mathrm{ex}(n,H)$, the maximum number of edges in an $n$-vertex $H$-free $r$-uniform hypergraph. With this application in mind, we find a range of new counterexamples to the conjecture for hypergraphs, including all linear hypergraphs containing a loose triangle and all $3$-partite $3$-uniform tight cycles.
\end{abstract}

\section{Introduction}

An \emph{$r$-graphon} $W:[0,1]^r \rightarrow [0,1]$ is an $r$-variable symmetric measurable function.\footnote{We note that this is different from the usual definition of \emph{hypergraphons} (see, for instance, \cite{ESz12}), where $2^r-2$ variables are used to model limits of $r$-uniform hypergraphs. Such an approach is required to make the space complete, which is not necessary for our purposes.} Given an $r$-uniform hypergraph (or \emph{$r$-graph}) $H$, the \emph{homomorphism density} $t_H(W)$ of $H$ in $W$ is 
\begin{align*}
    t_H(W) := \int \prod_{u_1\cdots u_r\in E(H)}W(x_{u_1},x_{u_2},\ldots,x_{u_r}) d\mu^{{\rm v}(H)}.
\end{align*}
An $r$-graph $H$ is said to be \emph{Sidorenko} if
\begin{align}\label{eq:Sidorenko}
    t_H(W) \geq t_{K_r}(W)^{{\rm e}(H)} = \left(\int W \, d\mu^r\right)^{{\rm e}(H)}
\end{align}
for all $r$-graphons $W:[0,1]^r \rightarrow [0,1]$, where $K_r$ denotes the $r$-graph with one edge. 
In graph-theoretic terms, an $r$-graph $H$ is Sidorenko if quasirandom $r$-graphs contain asymptotically the minimum number of copies of $H$ taken over all $r$-graphs with the same order and edge density. 
A celebrated conjecture of Sidorenko~\cite{Sid92, Sid93} (see also the closely related conjecture of Erd\H{o}s and Simonovits~\cite{ES84}) says that a graph $H$ is Sidorenko if and only if it is bipartite. The necessity of the bipartiteness condition is straightforward to verify, but its sufficiency remains wide open despite significant attention in recent years~\cite{CFS10, CKLL15, CL17, CL21, C22, CR21, KLL16, LSz12, Sz15}.

It is also tempting to make an analogous conjecture for $r$-uniform hypergraphs $H$, namely, that $H$ is Sidorenko if and only if it is $r$-partite. Unfortunately, as already observed in~\cite{Sid93}, this is false, with the $3$-uniform \emph{loose triangle} with vertex set $\{1, 2, \dots, 6\}$ and edges $\{1, 2, 3\}$, $\{3, 4, 5\}$, $\{5, 6, 1\}$ being a counterexample. However, as we will show, there is still something to be gained if the conjecture fails to hold, in that we can improve the lower bound for the extremal number of any $r$-uniform hypergraph $H$ for which Sidorenko's conjecture is false.

Given a natural number $n$ and an $r$-graph $H$, the \emph{extremal number} $\ext(n,H)$ is the maximum number of edges in an $n$-vertex $H$-free $r$-graph. 
It is known that for any fixed $r$-graph $H$, there exists a non-negative number $\pi(H)$ such that $\ext(n,H) = (\pi(H) + o(1)) \binom{n}{r}$ and that $\pi(H) = 0$ if and only if $H$ is $r$-partite. 
With very few exceptions (see, for example,~\cite{FS13} for classical results and~\cite{CJ22} and its references for more recent developments), 
the problem of estimating $\ext(n,H)$ more accurately in the degenerate case where $H$ is $r$-partite is wide open. In general, the best known lower bound comes from a simple application of the probabilistic deletion method and says that for any fixed $r$-partite $r$-graph $H$ there exists some constant $\gamma > 0$ such that
\[\ext(n,H) \ge \gamma n^{r-\frac{{\rm v}(H)-r}{{\rm e}(H)-1}}.\]
Our first result improves this estimate for non-Sidorenko $r$-graphs.

\begin{theorem}\label{thm:extremal_number}
For any non-Sidorenko $r$-graph $H$, there exist constants $c, \gamma >0$ such that 
\begin{align*}
    \ext(n,H) \ge \gamma n^{r-\frac{{\rm v}(H)-r}{{\rm e}(H)-1}+c}.
\end{align*} 
\end{theorem}

One reason this result is interesting is that, by a result of Ferber, McKinley and Samotij~\cite[Theorem 9]{FMS20}, any polynomial gain over the deletion  bound for the extremal number of an $r$-graph~$H$ implies an optimal counting result for the number of $H$-free $r$-graphs on $n$ vertices. Thus, we have the following corollary of~\cref{thm:extremal_number}.

\begin{corollary}\label{cor:count}
For any non-Sidorenko $r$-graph $H$, there exists $C > 0$ and an infinite sequence of positive integers $n$ such that
\[|\mathcal{F}_n(H)| \leq 2^{C \cdot \ext(n, H)},\]
where $\mathcal{F}_n(H)$ is the set of all labelled $H$-free $r$-graphs with vertex set $\{1, 2, \dots, n\}$.
\end{corollary}

We note in passing that results similar to \cref{thm:extremal_number} and \cref{cor:count} were obtained recently by Conlon, Pohoata and Zakharov~\cite{CPZ21} for $H = K_{2,2,\dots,2}$, the complete $r$-partite $r$-graph with two vertices in each part. However, since Sidorenko's conjecture does hold for these graphs through some standard applications of the Cauchy--Schwarz inequality, their proof proceeds along very different lines, making use of a multilinear variant of Bukh's random algebraic method~\cite{B15}.

Motivated by \cref{thm:extremal_number} and its application \cref{cor:count}, much of this paper is devoted to finding examples of $r$-partite $r$-graphs for which Sidorenko's conjecture is false. For instance, if we define the $r$-uniform \emph{loose triangle} to be the $r$-graph with vertex set $\{1, 2, 3, \dots, 3r-3\}$ and edges $\{1, 2, \dots, r\}$, $\{r, r+1, \dots, 2r-1\}$, $\{2r-1, \dots, 3r-3, 1\}$, then we have the following result. Note that here a \emph{linear $r$-graph} is an $r$-graph where every pair of edges intersect in at most one vertex.

\begin{theorem}\label{thm:linear}
    Any linear $r$-graph that contains a loose triangle is not Sidorenko.
\end{theorem}

By the celebrated $(6,3)$-theorem of Ruzsa and Szemer\'edi~\cite{RS78}, which states that dense linear $r$-graphs contain loose triangles, we have the following corollary.

\begin{corollary}\label{cor:RSz}
For any integer $r \geq 3$ and any $c > 0$, there exists $k_0$ such that any linear $r$-graph with $k \geq k_0$ vertices and at least $c k^2$ edges is not Sidorenko.
\end{corollary}

While the extremal number is known exactly for some sparse linear $r$-graphs such as loose paths and cycles~\cite{KMV15}, these results, applied in conjunction with~\cref{thm:extremal_number}, give the first polynomial improvement on the lower bound for the extremal number of a broad range of linear $r$-graphs.

In a somewhat different direction, we look at the tight cycles $C_\ell^{(r)}$ with vertex set $\{1, 2, \dots, \ell\}$ and edges $\{i, i+1, \dots, i+r-1\}$ for all $i = 1, 2, \dots, \ell$, where addition is taken mod $\ell$. From an extremal viewpoint, these are some of the most closely studied  hypergraphs (see, for example,~\cite{KLP22,L23,MPS,Raz10}). We will show that, at least for certain choices of $\ell$ and $r$, they are again not Sidorenko. In the statement below, we also consider the $r$-graphs $C_\ell^{(r)} - e$ obtained by deleting a single edge $e$ from $C_\ell^{(r)}$.

\begin{theorem}\label{thm:tight_cycles}
$C_k^{(3)}$ is not Sidorenko for any $k \geq 4$, $C_k^{(3)} - e$ is not Sidorenko for any $k \geq 7$ and $C_{2r}^{(r)}$ is not Sidorenko for any odd $r \geq 3$.
\end{theorem}

There are some recent results~\cite{BL23,KLP22} that determine the Tur\'an densities of $C_k^{(3)}$ and $C_k^{(3)}-e$ when $k$ is sufficiently large and \emph{not} divisible by $3$. 
\cref{thm:tight_cycles} gives the first non-trivial improvement on the lower bounds for the extremal numbers of $C_k^{(3)}$ and $C_k^{(3)}-e$ when $k$ \emph{is} divisible by $3$.

We also give some examples of $r$-graphs with the stronger property that they are not common. 
By saying that an $r$-graph $H$ is \emph{common}, we mean that $t_H(W) + t_H(1-W) \geq 2^{1-{\rm e}(H)}$ 
for every $r$-graphon $W:[0,1]^r \rightarrow [0,1]$. 
In graph-theoretic terms, an $r$-graph $H$ is common if the number of monochromatic copies of $H$ in a two-colouring of the edges of $K_n^{(r)}$ is asymptotically minimised by a quasirandom colouring. The study of such graphs is a central topic in Ramsey theory and we refer the interested reader to~\cite{KVW22} for further context and additional references. For us, the important point is that 
if an $r$-graph is Sidorenko, 
it is automatically common, so non-common $r$-graphs are automatically not Sidorenko. As it involves some further notation, 
we will hold off on giving a full description of our main result in this direction until Section~\ref{sec:common} and instead give an illustrative example.

\begin{theorem}\label{thm:grids}
For $r$ odd, the grid $r$-graph $G_r$ whose vertices are the points of the $r \times r$ grid 
and whose edges are the $2r$ horizontal and vertical lines of the grid is not common. 
\end{theorem}

Unlike our previous results, this does not allow us to give an improved bound for $\ext(n, G_r)$, since, by considering all of the edges containing a fixed vertex, we get the simple lower bound $\ext(n, G_r) \ge \binom{n-1}{r-1}$, which is considerably better than the deletion bound. However, the grid graphs are an interesting and well-studied family (see, for example,~\cite{FR13, GS22, Sid22}), so we believe the fact that its odd members are not common is an interesting result in its own right.

\section{Lower bounds for the extremal number}

In this short section, we will use the tensor power trick to prove Theorem~\ref{thm:extremal_number}, the statement that the deletion bound may be improved for counterexamples to Sidorenko's conjecture. We will need the following standard result from the theory of graph limits (see, for example,~\cite{L12}), obtained by sampling $n$ vertices $v_1, v_2, \dots, v_n$ independently and uniformly at random from $[0,1]$ and placing an edge on each $v_{i_1}, v_{i_2}, \dots, v_{i_r}$ with $i_1 < i_2 < \dots < i_r$ independently with probability $W(v_{i_1}, v_{i_2}, \dots, v_{i_r})$.

\begin{lemma}\label{lem:discretise}
Let $W$ be an $r$-graphon. Then there exists a sequence $(G_n)_{n=1}^{\infty}$ of $r$-graphs such that $|V(G_n)|=n$ and $t_F(G_n)$ converges to $t_F(W)$ for every fixed $r$-graph $F$.
\end{lemma}

The \emph{tensor product} $G\otimes H$ of two $r$-graphs $G$ and $H$ is the graph with vertex set $V(G) \times V(H)$ where $((x_1, y_1), (x_2, y_2), \dots, (x_r, y_r)) \in E(G\otimes H)$ if and only if $(x_1, x_2, \dots, x_r) \in E(G)$ and $(y_1, y_2, \dots, y_r) \in E(H)$. For $N$ a positive integer, we may then define $G^{\otimes N}$ inductively by $G^{\otimes 1} = G$ and $G^{\otimes N} = G \otimes G^{\otimes N-1}$. A key property of these tensor powers, which we will need below, is that $t_H(G^{\otimes N}) = t_H(G)^N$ for any graphs $G$ and $H$ and any positive integer $N$.

\begin{proof}[\bf{Proof of \cref{thm:extremal_number}}]
Let $W$ be an $r$-graphon for which $t_H(W) < t_{K_r}(W)^{{\rm e}(H)}$. If $(G_m)_{m=1}^{\infty}$ is a sequence of $r$-graphs with $|V(G_m)| = m$ given by~\Cref{lem:discretise}, then, provided $m$ is sufficiently large, $t_H(G_m) < t_{K_r}(G_m)^{{\rm e}(H)}$. Let $G$ be an $r$-graph from this sequence for which this is the case and let $\alpha_0:=t_{K_r}(G)=r! {\rm e}(G)/{\rm v}(G)^{r}$ and $\beta_0:=t_H(G)$, so that $\beta_0 < \alpha_0^{{\rm e}(H)}$. We will assume that $G$ is taken sufficiently large that $\beta_0/\alpha_0 \geq {\rm v}(G)^{r - {\rm v}(H)}$.

Set $n:={\rm v}(G)^N$, $\alpha := \alpha_0^N$ and $\beta:=\beta_0^N$. Then $G^{\otimes N}$ is an $n$-vertex graph with $\alpha n^r/r!$ edges and at most $\beta n^{{\rm v}(H)}$ labelled copies of $H$. Let $c' := \frac{{\rm e}(H)\log \alpha_0 - \log \beta_0}{\log {\rm v}(G)}$, so that $c' > 0$ and $\beta = \alpha^{{\rm e}(H)} n^{-c'}$. Crucially, the number of copies of $H$ in $G^{\otimes N}$ is significantly smaller than the random count of roughly $\alpha^{{\rm e}(H)}n^{{\rm v}(H)}$, allowing us to apply the deletion method more efficiently. 

Indeed, if we take a random subgraph $(G^{\otimes N})_p$ of $G^{\otimes N}$ where every edge appears independently with probability $p$, the expected number of edges $X$ in this subgraph is $p \alpha n^r/r!$ and the expected number of copies $Y$ of $H$ is at most $(p \alpha)^{{\rm e}(H)} n^{{\rm v}(H) -c'}$. Note that the condition $\beta_0/\alpha_0 \geq {\rm v}(G)^{r - {\rm v}(H)}$ is equivalent to $\alpha_0^{{\rm e}(H)-1} \geq {\rm v(G)}^{r+c'-{\rm v}(H)}$, which in turn implies that $\alpha^{{\rm e}(H)-1} \geq n^{r+c'-{\rm v}(H)}$. Therefore, there is some $p < 1$ such that $(p \alpha)^{{\rm e}(H)-1} = n^{r+c'-{\rm v}(H)}/(2 r!)$. But then 
\[\mathbb{E}[Y] \leq (p \alpha)^{{\rm e}(H)} n^{{\rm v}(H) -c'} = \frac{p \alpha n^r}{2 r!} = \frac{1}{2} \mathbb{E}[X],\]
so, by linearity of expectation, $\mathbb{E}[X-Y] \geq \frac 12 \mathbb{E}[X] = \frac{p \alpha n^r}{2 r!}$. Therefore, there must exist a graph for which we can delete an edge from every copy of $H$ and still leave at least $\frac{p \alpha n^r}{2 r!} \geq \gamma n^{r-\frac{{\rm v}(H)-r}{{\rm e}(H)-1}+c}$ edges, where $\gamma>0$ is an absolute constant and $c = c'/({\rm e}(H)-1)$. This yields the required conclusion when $n$ is a power of ${\rm v}(G)$, but, at the possible expense of replacing $\gamma$ with a smaller number, we can easily interpolate between these values.
\end{proof}

\section{Non-Sidorenko hypergraphs} \label{sec:nonsid}

\subsection{Linear hypergraphs}

Recall that an $r$-graph is \emph{linear} if every pair of edges shares at most one vertex. The \emph{girth} of a linear $r$-graph is the length of the shortest (loose) cycle in the graph. We shall prove the following statement that slightly generalises~\Cref{thm:linear}. 
In the proof, we will also need to know that, for $s\leq r$, the \emph{$(s-1)$-skeleton} of an $r$-graph $H$ is the $s$-graph obtained by replacing each $r$-edge of $H$ by a copy of $K_r^{(s)}$, the complete $s$-graph on $r$ vertices, and simplifying multiedges.

%When considering the $r$-graph homomorphism density from $H$ to $G$, a well-known technique in the theory of graph limits is to give weights to the edges.

For convenience, we will phrase our proof in the language of weighted $r$-graphs $f$ rather than $r$-graphons $W$. For us, this simply means that we are working with symmetric functions of the form $f:\Omega^r\rightarrow\mathbb{R}_{\geq 0}$, where $\Omega$ is a finite set. The functional $t_H(f)$ is then given by replacing the integral over $[0,1]^{{\rm v}(H)}$ by the sum over $\Omega^{{\rm v}(H)}$. We also allow $f$ to take values larger than $1$, but, since the inequality~\eqref{eq:Sidorenko} is homogeneous, it is sufficient for $f$ to be bounded and non-negative.

%A \emph{weighted $r$-graph} is a function $f:\Omega^r\rightarrow\mathbb{R}$ defined on the measure space $(\Omega,\PP)$ with a suitable probability measure $\PP(.)$. Otherwise specified, the probability measure in what follows will be the uniform measure if $\Omega$ is finite and the Lebesgue measure if $\Omega=[0,1]$. In the theory of hypergraph limits, the limits of $r$-graphs, the so-called \emph{hypergraphons}, have more variables than $r$ to obtain compactness, but this is not necessary for our purpose. We also remark that it is easy to convert $t_H(f)$ for a weighted $r$-graph $f$ to the usual (discrete) $r$-homomorphism density, especially for bounded and non-negative $f$.

\medskip

\begin{theorem}\label{thm:linear_girth}
If $H$ is a linear $r$-graph of odd girth, then $H$ is not Sidorenko.
\end{theorem}

\begin{proof}[\bf{Proof}]
Consider the weighted $r$-graph on $\{-1, 1\}$ where the edge $(x_1, x_2, \dots, x_r) \in \{\pm 1\}^r$ receives the weight $f(x_1,\dots,x_r) = 1-c\sum_{i<j}x_ix_j$. For $c \leq 1/\binom{r}{2}$, $f$ is a non-negative symmetric function with $t_{K_r}(f) = 1$.
%\footnote{%For convenience, we phrase our example in terms of weighted $r$-graphs $f$ rather than $r$-graphons $W$.
%, but it is easy to convert between the two settings. 
%Note that we allow $f$ to take values larger than $1$, but, since the inequality~\eqref{eq:Sidorenko} is homogeneous, it is sufficient for $f$ to be bounded and non-negative.} 
Observe that a monomial in the expansion of $\prod_{v_1\dots v_r \in E(H)}f(x_{v_1},\dots,x_{v_r})$, the integrand of $t_H(f)$, has zero average whenever it contains a variable of odd degree. Thus, the non-vanishing terms in the expansion of $t_H(f)$ correspond to `Berge' even subgraphs~$F$ of the $1$-skeleton of~$H$, those subgraphs~$F$ where every vertex has even degree and every (2-)edge $e\in E(F)$ extends to a unique $r$-edge in~$H$. Moreover, every such $F$ receives the weight $(-c)^{{\rm e}(F)}$. Therefore, if $g$ is the girth of $H$ and $g$ is odd, $t_H(f) = 1 -Kc^g + O(c^{g+1})$, where $K$ denotes the number of shortest loose cycles. By choosing $c>0$ small enough, we see $t_H(f)<1 = t_{K_r}(f)^{{\rm e}(H)}$, so $H$ is not Sidorenko.
\end{proof}

It is also possible to generalise~\cref{thm:linear_girth}, although we did not find any concrete applications of this more general result. Indeed, by replacing $f$ with $f(x_1,\dots,x_r) = 1-c\sum_{i_1< \dots < i_s} x_{i_1} \cdots x_{i_s}$ for any $s \le r$, one can show that if the smallest subgraph $F$ of the $(s-1)$-skeleton of $H$ where every vertex has even degree and every $s$-edge $e \in E(F)$ extends to a unique $r$-edge in $H$ has an odd number of edges, then $H$ is not Sidorenko. Since, in an $s$-uniform hypergraph $F$, $\sum_{v \in V(F)} d(v) = s \cdot {\rm e}(F)$, such a subgraph $F$ can only exist when $s$ is even. 

\subsection{Tight cycles}

Recall that $C_{\ell}^{(r)}$ denotes an $r$-uniform tight cycle of length $\ell$. Since $C_\ell^{(r)}$ and $C_\ell^{(r)} - e$ can only be $r$-partite when $\ell$ is a multiple of $r$, in order to prove Theorem~\ref{thm:tight_cycles}, it will suffice to study tight cycles of the form $C_{kr}^{(r)}$.

Given an $r$-graph $H$, let $\kappa_m(H)$ denote the number of subgraphs of $H$ with $m$ edges and no degree-one vertices. 
As captured by the following proposition, the polynomial $P_H(x) := \sum_{i=1}^{{\rm e}(H)} \kappa_i(H) x^i$ will play an important role in the proof of~\Cref{thm:tight_cycles}.

\begin{proposition}\label{th:root}
Let $r$ be odd and $H$ be a subgraph of $C_{kr}^{(r)}$. 
If $H$ is Sidorenko, then $P_H(x) \geq 0$ for all $x \in [-1,0]$. 
\end{proposition}

\begin{proof}[\bf{Proof}]
Suppose that $P_H$ takes a negative value on $[-1,0]$. 
Then there exists $c \in (0,1)$ 
such that $P_H(-c) < 0$. 
For $\varepsilon\in(0,1)$, let $f_\varepsilon$ be the function on $[0,1]$ defined by
\begin{align*}
    f_\varepsilon(x) = \begin{cases}
        \varepsilon ~~~~\text{ if } x\leq \frac{1}{1+\varepsilon}\\
        -1 ~~\text{ otherwise.}
    \end{cases}
\end{align*}
Then $\int_0^1 f_\varepsilon d\mu=0$ and, for any fixed integer $d>1$ and $\varepsilon$ sufficiently small, 
$\int_0^1 (f_\varepsilon)^d d\mu = (-1)^d \varepsilon + O(\varepsilon^2)$. 

Let
$g_\varepsilon(x_1,\dots,x_r) := \prod_{i=1}^r f_\varepsilon(x_i)$, so that
$t_G(g_\varepsilon)=0$ whenever $G$ has a vertex of degree one.
Moreover, for every $n$-vertex $r$-graph $G$ with degree sequence $d_1,\dots,d_n\geq 2$,
\begin{align*}
    t_G(g_\varepsilon)  = (-1)^{\sum_{i=1}^nd_i}\varepsilon^{{\rm v}(G)} +O(\varepsilon^{{\rm v}(G)+1}) =  (-1)^{{\rm e}(G)}\varepsilon^{{\rm v}(G)} +O(\varepsilon^{{\rm v}(G)+1}),
\end{align*}
since $r\cdot{\rm e}(G)=\sum_{i=1}^n d_i$ and ${\rm e}(G)$ have the same parity by the assumption that $r$ is odd.

Let $h_{\varepsilon,c} := 1 + cg_\varepsilon$, noting that this function is non-negative. 
By expanding out the expression $\prod_{v_1\dots v_r \in E(H)} h_{\varepsilon,c}(x_{v_1},\dots,x_{v_r}) = \prod_{v_1\dots v_r \in E(H)}(1+cg_{\varepsilon}(x_{v_1},\dots,x_{v_r}))$, we see that
\begin{align}\label{eq:h_expand}
    t_H(h_{\varepsilon,c}) = 1 + \sum_{G \subseteq H} c^{{\rm e}(G)} t_G(g_\varepsilon),
\end{align}
where the sum is taken over all non-empty edge subsets of $H$, which can be seen as subgraphs $G$ of $H$. 

In any subgraph of the tight cycle $C_{kr}^{(r)}$, 
degrees of consecutive vertices of the cycle differ by at most one.
Thus, in a non-empty subgraph $G$ with no vertices of degree one, 
no isolated vertices exist and, hence, all but those $G$ with minimum degree at least two vanish in~\eqref{eq:h_expand}. 
Therefore, $\kappa_m(H)$ counts the number of $m$-edge subgraphs of $H$ on $kr$ vertices with minimum degree at least two. It then follows that
\begin{align*}
    t_H(h_{\varepsilon,c}) &= 1 + \sum_{G \subseteq H, \, \delta(G)\geq 2} c^{{\rm e}(G)} t_G(g_\varepsilon)\\
    &=1 + \varepsilon^{kr}\sum_{G \subseteq H, \, \delta(G)\geq 2} (-c)^{{\rm e}(G)}  + O(\varepsilon^{kr+1}) = 1+\varepsilon^{kr}P_H(-c) +O(\varepsilon^{kr+1}).
\end{align*}
Therefore, for sufficiently small $\varepsilon>0$, 
$t_H(h_{\varepsilon,c}) < 1$. But $\int h_{\varepsilon,c} \, d\mu^r = 1$, so this contradicts our assumption that $H$ is Sidorenko. 
\end{proof}

We will now use \cref{th:root} to prove the following three results, which together make up~\cref{thm:tight_cycles}. 

\begin{theorem}\label{th:Pk3}
$C_{3k}^{(3)}$ is not Sidorenko for $k \geq 2$. 
\end{theorem}

\begin{theorem}\label{th:Pk3-e}
$C_{3k}^{(3)}-e$ is not Sidorenko for $k \geq 3$. 
\end{theorem}

\begin{theorem}\label{th:P2r}
$C_{2r}^{(r)}$ is not Sidorenko for any odd $r \geq 3$.
\end{theorem}

For the proofs, we will need to better understand the functions $\kappa_m(H)$ for the $r$-graphs $H$ under consideration.

\begin{lemma}\label{th:Pk3c}
$\kappa_i(C_{3k}^{(3)})=0$ for $i<2k$ and 
$\kappa_{2k+i}(C_{3k}^{(3)}) = \frac{3k}{k+2i} \binom{k+2i}{3i}$ 
for $0 \leq i \leq k$. 
\end{lemma}

\begin{proof}[\bf{Proof}]
A subgraph $G$ of $C_{3k}^{(3)}$ with $i$ edges 
such that each vertex has degree $2$ or $3$ 
must be obtained from $C_{3k}^{(3)}$ by removing 
$3k-i$ disjoint edges.
But the number of disjoint edges cannot exceed $k$ and 
the number of ways to select $1 \leq j \leq k$ 
independent edges in $C_{3k}^{(3)}$ is
$\frac{3k}{j} \binom{3k-1-2j}{j-1}$. 
Thus, $\kappa_i(C_{3k}^{(3)})=0$ for $i<2k$ and
\begin{align*}
  \kappa_{2k+i}(C_{3k}^{(3)}) & = 
  \frac{3k}{k-i} \binom{3k-1-2(k-i)}{k-i-1}
  = \frac{3k}{k-i} \binom{k+2i-1}{k-i-1}
  \\ &
  = \frac{3k}{k+2i} \binom{k+2i}{k-i}
  = \frac{3k}{k+2i} \binom{k+2i}{3i}
\end{align*}
for $0 \leq i \leq k-1$. 
Moreover, $\kappa_{3k}(C_{3k}^{(3)})=1$. 
\end{proof}

\begin{lemma}\label{th:Pk3-ec}
$\kappa_{i}(C_{3k}^{(3)}-e)=0$ for $i<2k$ 
and $\kappa_{2k+i}(C_{3k}^{(3)}-e) = \binom{k+2i-1}{3i}$ 
for $0 \leq i \leq k-1$. 
\end{lemma}

\begin{proof}[\bf{Proof}]
The statement follows from \cref{th:Pk3c} and 
the fact that 
$\kappa_i(C_{kr}^{(r)}-e) = \frac{kr-i}{kr} \kappa_i(C_{kr}^{(r)})$. 
\end{proof}

We also need to verify some elementary inequalities.

\begin{lemma}\label{th:lem}
For integers $k \geq 2$ and $i\geq 1$,
\[
  \frac{(k+2i+1)(k+2i)(k-i)}{(3i+3)(3i+2)(3i+1)} 
  \leq \frac{k^3+k^2}{60}.
\]
\end{lemma}

\begin{proof}[\bf{Proof}]
It is easy to check that each of the ratios 
$\frac{k+2i+1}{3i+3}$, 
$\frac{k+2i}{3i+2}$, 
$\frac{k-i}{3i+1}$ 
decreases with $i$. 
Hence,
\[
  \frac{(k+2i+1)(k+2i)(k-i)}{(3i+3)(3i+2)(3i+1)} 
  \leq \frac{(k+3)(k+2)(k-1)}{120} 
  = \frac{k^3 + 4k^2 + k - 6}{120}.
\]
We therefore need to show that 
$k^3 + 4k^2 + k - 6 \leq 2(k^3+k^2)$, 
which is equivalent to 
$F(k) := k^3 - 2k^2 - k + 6 \geq 0$. 
But this follows since $F(2)=4$ 
and $F'(k) = 3 k^2 - 4k - 1 = 3k(k-2) + 2k - 1 > 0$ 
for $k \geq 2$.
\end{proof}

\begin{lemma}\label{th:lem2}
For integers $k \geq 3$ and $i\geq 1$,
\[
  \frac{(k+2i+1)(k+2i)(k-i-1)}{(3i+3)(3i+2)(3i+1)} 
  \leq \frac{7}{600} (k^3-k).
\]
\end{lemma}

\begin{proof}[\bf{Proof}]
It is easy to check that each of the ratios 
$\frac{k+2i+1}{3i+3}$, 
$\frac{k+2i}{3i+2}$, 
$\frac{k-i-1}{3i+1}$ 
decreases with $i$. 
Hence,
\[
  \frac{(k+2i+1)(k+2i)(k-i-1)}{(3i+3)(3i+2)(3i+1)} 
  \leq \frac{(k+3)(k+2)(k-2)}{120} 
  = \frac{k^3 + 3k^2 - 4k - 12}{120}.
\]
We therefore need to show that 
$k^3 + 3k^2 - 4k - 12 \leq \frac{7}{5} (k^3-k)$, 
which is equivalent to 
$F(k) := 2k^3 - 15k^2 + 13k + 60 \geq 0$. 
But this follows since $F(3)=18$, $F(4)=F(5)=0$ 
and $F'(k) = 6k^2 - 30k  + 13 = 6k(k-5) + 13 > 0$ 
for $k \geq 5$.
\end{proof}

We are already in a position to prove Theorems~\ref{th:Pk3} and \ref{th:Pk3-e}.

\begin{proof}[\bf{Proof of \cref{th:Pk3}}]
By \cref{th:root}, it will be sufficient to find some 
$x \in [-1,0]$ such that $P_{C_{3k}^{(3)}}(x) < 0$. 
The coefficients of $P_{C_{3k}^{(3)}}$ are given by \cref{th:Pk3c}. 
It is easy to check that 
$P_{C_{6}^{(3)}}(x) = x^4 (3+6x+x^2)$ 
is negative at $x=-\frac{2}{3}$. 
Thus, we may assume that $k \geq 3$. 
For a fixed $k$ and $1 \leq i \leq k$, set 
\[ 
  A_i \, := \, \frac{3k}{k+2i} \binom{k+2i}{3i} \,
  \Big(\frac{30}{k^3+k^2}\Big)^i .
\]
By \cref{th:lem}, for $1 \leq i \leq k-1$, 
\begin{align*}
  \frac{A_{i+1}}{A_i} \: = \: 
  \frac{(k+2i+1)(k+2i)(k-i)}{(3i+3)(3i+2)(3i+1)} 
  \cdot \frac{30}{k^3+k^2}
  \:\leq\: \frac{1}{2}.
\end{align*}
Set $x=\frac{30}{k^3+k^2}$. 
As $A_{2j} \leq \frac{1}{2} A_{2j-1} \leq A_{2j-1}$, 
we get 
\begin{align*}
  x^{-2k} P_{C_{3k}^{(3)}}(-x) \: & = \: 3 + \sum_{i=1}^k (-1)^i A_i 
  \:\leq\: 3 - A_1 + A_2
  \:\leq\: 3 - \frac{1}{2} A_1 
  \\ & = \: 3 - \frac{1}{2} \, \frac{3k}{k+2} \, \binom{k+2}{3} \, \frac{30}{k^3+k^2} 
  \: = \: 3 - \frac{15}{2} \: < \: 0,
\end{align*}
as required.
\end{proof}

\begin{proof}[\bf{Proof of \cref{th:Pk3-e}}]
By \cref{th:root}, 
it will be sufficient to find some 
$x \in [-1,0]$ such that $P_{C_{3k}^{(3)}-e}(x) < 0$. 
The coefficients of $P_{C_{3k}^{(3)}-e}$ are given by \cref{th:Pk3-ec}. 
It is easy to check that 
$P_{C_{9}^{(3)}-e}(x) = x^6 (1+4x+x^2)$ 
is negative at $x=-\frac{2}{3}$. 
Thus, we may assume $k \geq 4$. 
For a fixed $k$ and $1 \leq i \leq k-1$, set 
\[ 
  B_i \, := \, \binom{k+2i-1}{3i} \, 
  \left(\frac{300}{7(k^3-k)}\right)^i. 
\]
By \cref{th:lem2}, for $1 \leq i \leq k-2$, 
\begin{align*}
  \frac{B_{i+1}}{B_i} \: = \: 
  \frac{(k+2i+1)(k+2i)(k-i-1)}{(3i+3)(3i+2)(3i+1)} 
  \cdot \frac{300}{7(k^3-k)}
  \:\leq\: \frac{1}{2}.
\end{align*}
Set $x = \frac{300}{7(k^3-k)}$. 
As $B_{2j} \leq \frac{1}{2} B_{2j-1} \leq B_{2j-1}$, 
we get 
\begin{align*}
  x^{-2k} P_{C_{3k}^{(3)}-e}(-x) \: & = \: 1 + \sum_{i=1}^{k-1} (-1)^i B_i 
  \:\leq\: 1 - B_1 + B_2
  \:\leq\: 1 - \frac{1}{2} B_1 
  \\ & = \: 1 - \frac{1}{2} 
    \binom{k+1}{3} \, \frac{300}{7(k^3-k)} 
  \: = \: 1 - \frac{25}{7} \: < \: 0,
\end{align*}
as required.
\end{proof}

For the proof of \cref{th:P2r}, we need to do a little more work. 
Consider an $m$-element subset $A\subseteq\ZZ_{2r}$ and assume that its elements are cyclically ordered as
$A=(x_0,x_1,\dots,x_m=x_0)$. 
We say that $x_i$ is \emph{good} if 
$x_{i+1}-x_{i-1} \in \{2,\dots,r\}$ 
and \emph{bad} otherwise. 

\begin{lemma}\label{th:bad}
The number of $m$-element subsets $A\subseteq\ZZ_{2r}$ 
that have at least one bad element is 
$2r(m-2)\binom{r}{m-1}$ for 
$m \geq 4$. 
\end{lemma}

\begin{proof}[\bf{Proof}]
Suppose $A=(x_0, x_1, \dots, x_m=x_0)$ is such a subset. 
Notice that if $x_i$ and $x_j$ are two distinct bad points, 
then $i-j = \pm 1$. 
Hence, there is either just one bad point or there are two consecutive bad points. 
Thus, there exists a unique index $j$ such that $x_j$ is good
and $x_{j-1}$ is bad. 
Without loss of generality, we may assume that $j=1$. 
$x_1$ can have any of the $2r$ possible values. 
We will assume that $x_1=0$ and show that there are then 
exactly $(m-2) \binom{r}{m-1}$ choices 
for $x_0,x_2,\dots,x_{m-1}$. 
As $x_1$ is good, $x_2-x_0 \in \{2,\dots,r\}$. 
As $x_0$ is bad, $x_1-x_{m-1} \in \{r+1,\dots,2r-1\}$, 
so $x_{m-1} \in \{1,\dots,r-1\}$. 
If $x_2=i$, there are $r-i$ choices for $x_0$ 
and $\binom{(r-1)-i}{m-3}$ choices for $x_3,\dots,x_{m-1}$. 
Thus, the total number of choices is 
\begin{align*}
  \sum_{i=1}^{r-1} (r-i) \binom{(r-1)-i}{m-3} 
  \: = \:
  \sum_{i=1}^{r-1} (m-2) \binom{r-i}{m-2}
  \: = \:
  (m-2) \binom{r}{m-1},
\end{align*}
as required.
\end{proof}

Note that a subgraph $H$ of the tight cycle $C_{2r}^{(r)}$ has no  degree-one vertices if and only if the set $A$ of initial vertices of edges in $H$ contains no bad elements. Thus, we have the following immediate corollary of~\cref{th:bad}.

\begin{corollary}\label{th:P2rc}
\[
  \kappa_i(C_{2r}^{(r)}) = 
  \begin{cases}
  \;\;\;\;\;\;\;\;\;\;\;\;\;\;\;\;\; 
  0 \;\;\;\;\;\;\;\;\;\;\;\;\;\;\;\;\,{\rm if}\;\; i \leq 3, \\
  \binom{2r}{i} - 2r(i-2)\binom{r}{i-1}
  \;\;{\rm if}\;\; 4 \leq i \leq 2r.
  \end{cases}
\]
\end{corollary}

\begin{proof}[\bf{Proof of \cref{th:P2r}}]
By \cref{th:P2rc}, since $2r(m-2)\binom{r}{m-1} = 0$ for $r+2 \leq m \leq 2r$,
\begin{align*}
  P_{C_{2r}^{(r)}}(x) \: = \: & 
  \sum_{i=4}^{2r} \binom{2r}{i} x^i \: - \:
  \sum_{i=4}^{r+1} 2r(i-2)\binom{r}{i-1} x^i
  \\ = & \:
  \sum_{i=0}^{2r} \binom{2r}{i} x^i \: - \:
  \left(1 + 2rx + r(2r-1) x^2 + \frac{2}{3} r(r-1)(2r-1) x^3\right) 
  \\ & \;\;\;\; 
  - 2r \sum_{i=1}^{r+1} (i-2)\binom{r}{i-1} x^i
  \: + \: 2r \left(-x + \frac{1}{2} r(r-1) x^3\right)
  \\ = & \:
  (1+x)^{2r}
  \: + \: 2r \sum_{i=1}^{r+1} \binom{r}{i-1} x^i
  \: - \: 2r \sum_{i=1}^{r+1} (i-1)\binom{r}{i-1} x^i
  \\ & \;\;\;\; 
  \: - \:
  \left(1 + 4rx + r(2r-1) x^2 + \frac{1}{3} r(r-1)(r-2) x^3\right) 
  \\ = & \:
  (1+x)^{2r} + 2rx(1+x)^r - 2r^2 x^2 (1+x)^{r-1}
  \\ & \;\;\;\; 
  \: - \:
  \left(1 + 4rx + r(2r-1) x^2 + \frac{1}{3} r(r-1)(r-2) x^3\right).
\end{align*}
If $r \geq 16$, set $x = -\frac{1}{r}$. 
Then $(1+x)^{2r} < e^{-2}$ and $(1+x)^r + (1+x)^{r-1} \geq 2e^{-1}$, so that 
\begin{align*}
  P_{C_{2r}^{(r)}}(x) \: < \: 
  e^{-2} - 4e^{-1} + \frac{4}{3} + \frac{2}{3r^2}
  \: < \: -0.002849 + \frac{2}{3r^2}
  \: < \: 0.
\end{align*}
If $r \leq 16$, set $x = -\frac{2}{r}$. 
Then $(1+x)^{2r} < e^{-4}$ and 
\begin{align*}
  4(1+x)^r + 8(1+x)^{r-1} \: = \:
  4((1+x)^r + (1+x)^{r-1}) + 4(1+x)^{r-1} \: > \:
  8e^{-2} + 4e^{-2} = 12e^{-2},
\end{align*}
so that
\begin{align*}
  P_{C_{2r}^{(r)}}(x) \: < \: 
  e^{-4} - 12e^{-2} + \frac{5}{3} - \frac{4}{r} + \frac{16}{3r^2}
  \: < \: 0.060959 - \frac{4}{r} + \frac{16}{3r^2}
  \: < \: 0.
\end{align*}
Therefore, by \cref{th:root}, $C_{2r}^{(r)}$ is not Sidorenko for any odd $r \geq 3$.
\end{proof}

To conclude this section, we note that we have also used \cref{th:root} to show that $C_{kr}^{(r)}$ is not Sidorenko for all values of $k$ and $r$ with $r \ge 5$ odd and $kr \le 30$. This suggests, and we conjecture, that $C_{kr}^{(r)}$ is not Sidorenko for any odd $r\ge 5$ and $k \geq 2$.

\section{Non-common hypergraphs} \label{sec:common}

Recall that an $r$-graph $H$ is \emph{common} if 
$t_H(W) + t_H(1-W) \geq 2^{1-{\rm e}(H)}$ 
for any $r$-graphon
$W:[0,1]^r \rightarrow [0,1]$ and that every Sidorenko 
hypergraph is automatically common.
By substituting $W=\frac{1+f}{2}$, 
we can rewrite the requirement for $H$ to be common as 
$t_H(1+f)+t_H(1-f) \geq 2$ 
for any $r$-variable symmetric measurable function 
$f:[0,1]^r \rightarrow [-1,1]$.
By expanding out, this inequality is equivalent to 
\begin{align}\label{eq:common}
  \sum_{G \subseteq H,\; {\rm e}(G) 
  \equiv 0 \:{\rm mod}\: 2,\; {\rm e}(G)>0
  } t_G(f) 
  \:\geq\: 0.
\end{align}
If \cref{eq:common} fails for some function $f$, 
then $H$ is not common and, hence, is not Sidorenko.  

To state the main result of this section, we need some definitions. 
Following Antol\'in Camarena et al.~\cite{ACHLL12}, we say that an $r$-graph $H$ is \emph{positive} if 
$t_H(W) \geq 0$ for any $r$-variable symmetric function $W:[0,1]^r \rightarrow [-1,1]$. When 
$r \geq 3$, we say that an $r$-graph is \emph{$2$-connected} if the removal of a single vertex or a single edge 
does not disconnect it, while, for $r = 2$, we just mean the usual notion, 
that a graph is $2$-connected if it is not disconnected by the removal of a single vertex. 

\begin{theorem}\label{th:even}
Let $r$ be odd or $r=2$. 
If an $r$-graph $H$ has a non-positive $2$-connected subgraph with $2m$ edges 
and every other subgraph with an even number of edges not exceeding $2m$ 
is either non-positive and $2$-connected or has a vertex of degree $1$, 
then $H$ is non-common. 
\end{theorem}

When $r \geq 3$, examples coming from \cref{th:even} are quite plentiful. 
To see this, we will make use of the following proposition
from~\cite{Sid22}. Here the \emph{Levi graph} of an $r$-graph $H$ 
is the bipartite graph $L(H)$ with vertex set $V(H) \cup E(H)$ 
where $v \in V(H)$ and $e \in E(H)$ are adjacent if and only if 
$v \in e$ in $H$. Note that for $r \ge 3$ the $r$-graph $H$ is $2$-connected if and only if its
Levi graph $L(H)$ is $2$-connected in the usual sense.

\begin{proposition}[\cite{Sid22}, Theorem~1.2]\label{th:Levi}
If an $r$-graph $H$ is positive, 
then its Levi graph $L(H)$ is positive. 
When $r$ is odd, 
$L(H)$ is positive if and only if $H$ is positive. 
\end{proposition}

\begin{example}
Consider the half-octahedron $G$, the $3$-graph 
with vertices $1,2,3,4,5,6$ and edges 
$\{1,3,5\}$, $\{1,4,6\}$, $\{2,3,6\}$, $\{2,4,5\}$. 
Its Levi graph $L(G)$ has $10$ vertices and $12$ edges. 
With a single exception, all positive graphs with at most $10$ vertices 
are classified in \cite{ACHLL12} 
and $L(G)$ is not one of them. 
Hence, by \cref{th:Levi}, $G$ is non-positive. 
Therefore, if all $4$-edge subgraphs of a $3$-graph $H$ 
are either isomorphic to $G$ or have a vertex of degree $1$,
then, by \cref{th:even}, $H$ is non-common and non-Sidorenko. In particular, by Theorem~\ref{thm:extremal_number}, the extremal number of the half-octahedron (or the Pasch configuration, as it sometimes known) is at least $\gamma n^{2 + c}$ for some $c, \gamma > 0$, improving on the bound of $\gamma n^2$ which follows from the deletion method or by taking all $\binom{n-1}{2}$ edges containing a given vertex. 
\end{example}

\begin{example} \label{exm:grid}
Recall that $G_r$ is the \emph{grid} $r$-graph 
whose vertices are the points of the $r \times r$ grid 
and whose edges are the $2r$ horizontal and vertical lines of the grid. 
It was shown in \cite[Proposition 1.5]{Sid22} that $G_r$ is not positive for odd $r$. 
Since any proper subgraph of $G_r$ has a vertex of degree $1$, 
\cref{th:even} implies that $G_r$ is non-common for odd $r$. 
Moreover, if we add more edges to $G_r$ without creating new subgraphs 
whose minimum vertex degree is at least $2$, 
then the resulting $r$-graph will remain non-common. 
\end{example}

The next statement was proved in \cite[Lemma~6]{ACHLL12} for $r=2$, 
but the proof can be repeated verbatim for an arbitrary $r$. 

\begin{lemma}[\cite{ACHLL12}]\label{th:disjoint}
An $r$-graph $G$ is positive if and only if every connected $r$-graph 
that occurs among the connected components of $G$ an odd number of times 
is positive. 
\end{lemma}

We also note the following result. 
In the proof, we often consider the \emph{tensor product} $f\otimes g$ of $r$-variable symmetric functions $f$ and $g$, defined by
\begin{align*}
    (f\otimes g)((x_1,y_1),\dots,(x_r,y_r))  = f(x_1,\dots,x_r) \, g(y_1,\dots,y_r),
\end{align*}
where we identify each $((x_1,y_1),\dots,(x_r,y_r))$ with a point in $[0,1]^r$ through a measure-preserving bijection from $[0,1]^{2r}$ to $[0,1]^r$.

\begin{proposition}\label{th:single_f}
If the $r$-graphs $G_1,\dots,G_k$ are not positive, 
then there exists a function $f$ 
such that $t_{G_i}(f) < 0$ for all $i=1,\dots,k$. 
\end{proposition}

\begin{proof}[\bf{Proof}]
We use induction on $k$. 
The base case $k=1$ is trivial, so we consider the induction step going from $k-1$ to $k \geq 2$. 
For each $j=1,\dots,k$, by the induction hypothesis, 
there exists a function $f_j$ such that 
$t_{G_i}(f_j) < 0$ for all $i \neq j$.
If necessary, we perturb $f_j$ by a little bit to ensure that $t_{G_j}(f_j) \neq 0$ 
while preserving $t_{G_i}(f_j) < 0$ for all $i \neq j$.
If $t_{G_j}(f_j) < 0$, then $f_j$ is the function we need. 
Thus, we may assume that $t_{G_j}(f_j) > 0$. 
If $k$ is even, 
$f=f_1 \otimes \cdots \otimes f_k$ 
satisfies $t_{G_i}(f) < 0$ for all $i=1,\dots,k$. 
Suppose then that $k$ is odd. 
Let $G$ be the union of disjoint copies of 
$G_1,\dots,G_k$. 
By \cref{th:disjoint}, $G$ is not positive, 
so there exists a function $f_0$ such that $t_G(f_0) < 0$. 
Notice that $t_G(f_0) = \prod_{i=1}^k t_{G_i}(f_0)$. 
We may assume that $t_{G_i}(f_0) > 0$ for $1 \leq i \leq m$ 
and $t_{G_i}(f_0) < 0$ for $m+1 \leq i \leq k$, 
where $m$ is even. 
If $m=0$, then $f_0$ is the function we need. 
If $m>0$, 
then $f=f_0\otimes f_1 \otimes \cdots \otimes f_m$ 
satisfies $t_{G_i}(f) < 0$ for all $i=1,\dots,k$, so we again have the required function. 
\end{proof}

The following result and its corollaries are the key to proving~\cref{th:even}. 
Note that we call an $r$-variate symmetric measurable function $f$ 
\emph{zero-averaging} if 
$\int f(x_1,\dots,x_{r-1},x_r) \, dx_r = 0$ 
for any $x_1,\dots,x_{r-1}$. 

\begin{proposition}\label{th:zero}
If $H$ is a non-positive $2$-connected graph with an even number of edges, then 
there exists a zero-averaging function $f:[0,1]^2\rightarrow [-1,1]$ 
such that $t_H(f)<0$. 
\end{proposition}
\begin{proof}
    Our plan is to construct a $\{\pm 1\}$-weighted graph $\Gamma$ such that every vertex has a vanishing sum over the weights of its incident edges and $t_H(\Gamma)<0$.
    Let $U:[0,1]^2\rightarrow [-1,1]$ be a measurable symmetric function, i.e., a signed graphon, that satisfies $t_H(U)<0$. 
    By using a signed variant of~\Cref{lem:discretise},  
    %the standard decomposition $U=U^+ - U^-$ and applying~\Cref{lem:discretise} to find appropriate graphs approximating the graphons $U^+$ and $U^-$, 
    one may obtain a $\{\pm 1\}$-weighted graph $G$ such that $t_H(G)<0$.
    Let $d:={\rm v}(G)$ and $s:={\rm v}(H)$ for brevity and write $w_H(G) = d^s t_H(G)$. We may then assume that $w_H(G)<-d^{s-1/2}$ by replacing $G$ with a blow-up by a sufficiently large factor. Note that, as $H$ has an even number of edges, $w_H(-G)=w_H(G)$, where $-G$ denotes the graph with edge weights of opposite sign from~$G$. 

    For any sufficiently large even $n$, there is a $d$-regular $n$-vertex bipartite graph $F$ with girth greater than $s$ (see, for example,  \cite{Lazebnik:1995}). Partition the edges of $F$ into $d$ perfect matchings, colouring the edges of each matching with one of the $d$ colours from $[d]$. 
    Now consider the line $d$-graph 
    of $F$ whose vertices are the edges of $F$ and whose $d$-edges are the collections of $d$ edges in $F$ incident to each vertex of $F$. 
    A $\{\pm 1\}$-weighted $d$-graph $\mathcal{F}$ is defined by assigning $+1$ or $-1$ to each edge in this line $d$-graph, depending on which side of the bipartition the corresponding vertex of $F$ lies.
    The required zero-averaging weighted graph $\Gamma$ is then obtained by replacing each $d$-edge in $\mathcal{F}$ by a copy of $G$, where we map each vertex of $V(G)=[d]$ to the corresponding coloured vertex (with the colouring inherited from the matchings) and multiply the weight on each edge of $G$ by the $\{\pm 1\}$-weight on the corresponding edge of $\mathcal{F}$.

    We claim that $t_H(\Gamma)<0$. To prove this, we say that a (weighted) homomorphism from~$H$ to $\Gamma$ is \emph{good} if the homomorphic image of $H$ lies in one $d$-edge of $\mathcal{F}$. The weighted sum of good homomorphisms is a negative number less than $n\cdot w_H(G)< -n d^{s-1/2}$, where we used the fact that $w_H(G)=w_H(-G)$.
    
    On the other hand, there may be some homomorphic images of $H$ that are not entirely covered by a single $d$-edge, which we call \emph{bad}. 
    As the girth of $F$ is larger than the number of vertices in $H$, the unique minimal collection of $d$-edges whose union contains a fixed bad image of $H$ must form a \emph{$d$-hypertree} in $\mathcal{F}$, where uniqueness follows from the fact that $\mathcal{F}$ is a linear hypergraph. 

    As $\mathcal{F}$ is linear, deleting a vertex $v$ lying in the intersection of two $d$-edges in a $d$-hypertree disconnects the subgraph of $\Gamma$ induced on the vertex set of the $d$-hypertree. In particular, if a bad image of $H$ contains $v$, 
    then the image of $H$ is disconnected once $v$ is deleted. As $H$ is a 2-connected graph, the bad image of $H$ must therefore be degenerate, i.e., there are at least two vertices of $H$ that are mapped to $v$. 

    Suppose that there are $t+1$ edges in a $d$-hypertree $\mathcal{T} \subseteq \mathcal{F}$ for some $t\geq 1$. Then there are $(t+1)d-t$ vertices in $\mathcal{T}$, exactly $t$ of which have degree two (here we used that each vertex in $\mathcal{F}$ has degree exactly two). If $\mathcal{T}$ is a minimal cover of a bad homomorphic image of $H$, then there are $t$ disjoint pairs of vertices in $H$, each of which maps to a unique one of the $t$ vertices of degree two. 
    Thus, there are at most $\binom{s}{2}^t((t+1)d-t)^{s-2t}$ bad homomorphic copies of $H$ whose minimal cover is $\mathcal{T}$. 
    
    Given a $d$-hypertree $\mathcal{T}$ with $t+1$ edges, one can recover the tree in $F$ corresponding to the union of the $d$-edges of $\mathcal{T}$ in three steps:  replace each $d$-edge by the corresponding vertex in $F$; connect those vertices that correspond to intersecting $d$-edges; and turn each degree-one vertex in a $d$-edge $e$ of $\mathcal{T}$ into a leaf adjacent to the unique vertex of $F$ which corresponds to $e$.
    Note that the leaves added in the last step are determined once the first two steps give a $(t+1)$-vertex tree in $F$. Thus, there are at most $nd^{t}$ isomorphic images of $\mathcal{T}$ in $\mathcal{F}$.
    Therefore, there are at most $s^{2t}(t+1)^{s-2t}nd^{s-t}$ bad homomorphic images of $H$ whose minimal cover is isomorphic to $\mathcal{T}$. Hence, as $1\leq t<s$ and the number of distinct $d$-hypertrees with $t+1$ edges and maximum degree two is bounded as a function of $t$, the number of bad homomorphic images of $H$ is at most $Cnd^{s-1}$ for a constant $C=C(s)$. This is asymptotically smaller than $-nd^{s-1/2}$, the upper bound for the weighted sum of good homomorphisms provided $d$ is sufficiently large, so that $t_H(\Gamma)$ is negative, as required.
\end{proof}

\begin{corollary}\label{th:zero_r}
Let $r \geq 3$ be odd. If $H$ is a non-positive $2$-connected $r$-graph with an even number of edges, 
then there exists an $r$-variate zero-averaging function $h$ 
such that $t_H(h) < 0$. 
\end{corollary}

\begin{proof}[\bf{Proof}]
By \cref{th:Levi}, the Levi graph $L(H)$ of $H$ is a non-positive $2$-connected graph with an even number of edges. 
Therefore, by \cref{th:zero}, 
there exists a $2$-variate zero-averaging function $f$ 
such that $t_{L(H)}(f) < 0$. 
Consider the $r$-variate symmetric measurable function $h$ given by 
$h(x_1,\dots,x_r) = \int \prod_{i=1}^r f(x_i,y) \, dy$. 
It is easy to see that $h$ is zero-averaging and 
$t_H(h) = t_{L(H)}(f) < 0$. 
\end{proof}

\begin{proposition}\label{th:common_f}
If $G_1,\dots,G_k$ are non-positive $2$-connected graphs each with an even number of edges, 
then there exists a zero-averaging function $f$ 
such that $t_{G_i}(f) < 0$ for all $i=1,\dots,k$. 
\end{proposition}

\begin{proof}[\bf{Proof}]
Using \cref{th:zero}, we proceed exactly as in the proof of \cref{th:single_f}. 
\end{proof}

\begin{proposition}\label{th:common_f_r}
Let $r \geq 3$ be odd. 
If $G_1,\dots,G_k$ 
are non-positive $2$-connected $r$-graphs each with an even number of edges, 
then there exists a zero-averaging function $f$ 
such that $t_{G_i}(f) < 0$ for all $i=1,\dots,k$. 
\end{proposition}

\begin{proof}[\bf{Proof}]
Using \cref{th:zero_r}, we proceed exactly as in the proof of \cref{th:single_f}. 
\end{proof}

We are now in a position to prove \cref{th:even}.

\begin{proof}[\bf{Proof of \cref{th:even}}]
Let us assume that $m$ in the statement of the theorem 
is as small as possible, 
that is, there are non-positive $2$-connected 
subgraphs $G_1,\dots,G_k$ with $2m$ edges 
and every other subgraph with an even number of edges not exceeding $2m$ 
has a vertex of degree $1$. 
By \cref{th:common_f} (if $r=2$) or \cref{th:common_f_r}  (if $r$ is odd), 
there exists a function $f$ such that 
$S := t_{G_1}(f) + \cdots + t_{G_k}(f) < 0$ and 
$t_G(f)=0$ for any $r$-graph $G$ that has a vertex of degree $1$. 
Hence, for $\varepsilon > 0$ sufficiently small,
\begin{align*}
  \sum_{G \subseteq H,\; {\rm e}(G) \equiv 0 \:{\rm mod}\: 2,\; {\rm e}(G)>0} t_G(\varepsilon f) 
  \: = \: \varepsilon^{2m} S + O(\varepsilon^{2m+1}) 
  \: < \: 0,
\end{align*}
so $H$ is not common.
\end{proof}

\section{Concluding remarks}

We say that an $r$-graph $H$ is \emph{locally Sidorenko} if there exists $\varepsilon > 0$ such that $t_H(W) \geq t_{K_r}(W)^{{\rm e}(H)}$ for all $r$-graphons $W$ with $\|W-1/2\|_\Box \leq \varepsilon$. That is, $H$ is locally Sidorenko if the required inequality holds for all $r$-graphons which are sufficiently close to the uniform graphon $1/2$, where closeness is measured in terms of the cut norm (see, for example,~\cite{L12}). Since it is probably difficult to give a complete characterisation of those $r$-graphs which are Sidorenko, we instead conclude by asking for a characterisation of locally Sidorenko $r$-graphs. 

\begin{question}
Which $r$-graphs are locally Sidorenko?
\end{question}

It was shown by Lov\'asz~\cite{Lov11} that every bipartite graph is locally Sidorenko and later, by Fox and Wei~\cite{FW17}, that a graph is locally Sidorenko if and only if it is either a forest or has even girth. The results of Section~\ref{sec:nonsid} are all proved by showing that the relevant $r$-graphs are not locally Sidorenko and may help give some hints as to what a full characterisation should look like. However, at present, we have no concrete conjectures, even in the $r$-partite case most relevant to us. Indeed, despite Theorem~\ref{thm:tight_cycles}, it is already open to determine which tight cycles are locally Sidorenko for $r \ge 4$.

\bibliographystyle{plainurl}
\bibliography{references}

\end{document}